\documentclass[11pt,reqno]{amsart}
\usepackage{pb-diagram, graphics}
\usepackage{setspace}
\usepackage{amscd,amsfonts}
\usepackage{amssymb}
\usepackage{amsmath}
\usepackage{amsthm}
\usepackage{hyperref}
\theoremstyle{plain}
\newtheorem{thm}{Theorem}[subsection]

\newtheorem{lem}[thm]{Lemma}
\newtheorem{prop}[thm]{Proposition}
\newtheorem{rem}[thm]{Remark}

\numberwithin{equation}{subsection}

\def\Z{{\mathbb Z}}
\def\comp{\textnormal{comp}}
\def\N{\mathbb N}
\def\C{\mathbb C}
\def\h{\mathfrak{h}}
\def\g{\mathfrak{g}}
\def\sl{\mathfrak{sl}}
\def\n{\mathfrak{n}}
\def\f{\mathcal{F}}
\def\supp{\text{supp }}
\def\m{\mathcal{M}}
\def\bu{\mathbf{U}}
\def\B{\mathbf{B}}
\def\fB{\mathfrak{B}}
\def\comp{\textnormal{comp}}
\def\span{\textnormal{span}}

\begin{document}
\normalsize

\title[Bases for the Global Weyl modules of $\sl_n$ of highest weight $m\omega_1$]{Bases for the Global Weyl modules of $\sl_n$ of highest weight $m\omega_1$}

\author{Samuel Chamberlin}
\address{Department of Information Systems, Computer Science and Mathematics\\
Park University\\
Parkville, MO 64152}
\email{samuel.chamberlin@park.edu}

\author{Amanda Croan}
\address{Department of Information Systems, Computer Science and Mathematics\\
Park University\\
Parkville, MO 64152}
\email{amanda.croan@park.edu}

\begin{abstract}
We utilize a theorem of B. Feigin and S. Loktev to give explicit bases for the global Weyl modules for the map algebras $\mathfrak{sl}_n\otimes A$ of highest weight $m\omega_1$. These bases are given in terms of specific elements of $\bu(\sl_n\otimes A)$ acting on the highest weight vector.
\end{abstract}

\maketitle

\section{Introduction}
Let $\g$ be a simple finite dimensional complex Lie algebra. For the loop algebras, $\g\otimes\C[t,t^{-1}]$, the global Weyl modules were introduced by Chari and Pressley, \cite{CP}.  Feigin and Loktev extended these global Weyl modules to the case where the Laurent polynomials above were replaced by the coordinate ring of a complex affine variety, \cite{FL}. Chari, Fourier and Khandai then generalized this definition to the map algebras, $\g\otimes A$, where $A$ is a commutative, associative complex unital algebra, \cite{CFK}. Feigin and Loktev also gave an isomorphism, which explicitly determines the structure of the global Weyl modules for the map algebras of $\sl_n$ of highest weight $m\omega_1$, \cite{FL}.

The goal of this work is to use the structure isomorphism given by Feigin and Loktev to give nice bases for the global Weyl modules for the map algebras of $\sl_n$, $\sl_n\otimes A$, of highest weight $m\omega_1$. These bases will be given in terms of specific elements of $\bu(\sl_n\otimes A)$ acting on the highest weight vector. This was done in \cite{Cham} in the case $n=2$, but the case $n>2$ has not previously appeared in the literature.

\section{Preliminaries}

\subsection{The Structure of $\sl_n$}

Recall that $\sl_n$ is the Lie algebra of all complex traceless matrices
The Lie bracket is the commutator bracket given by $[A,B]=AB-BA$.

Given any matrix $\left[b_{i,j}\right]$ define $\varepsilon_k\left(\left[b_{i,j}\right]\right):=b_{k,k}$. For $i\in\{1,\ldots,n-1\}$ define $\alpha_i:=\varepsilon_i-\varepsilon_{i+1}$. Define $$R^\pm:=\left\{\displaystyle\pm\left(\alpha_i+\dots+\alpha_j\right)\bigg| 1\leq i<j\leq n-1\right\}$$ to be the positive and negative roots respectively, and define $R=R^+\cup R^-$ to be the set of roots. Let $e_{i,j}$ be the $n\times n$ matrix with a one in the $i$th row and $j$th column and zeros in every other position. Define $h_i:=h_{\alpha_i}=e_{i,i}-e_{i+1,i+1}$, for $i\in\{1,\ldots,n-1\}$. Then $\h:=\span\{h_i|1\leq i\leq n\}$ is a Cartan sub-algebra of $\sl_n$. Given $\alpha=\alpha_i+\cdots+\alpha_j\in R^+$ define $x_\alpha:=e_{i,j}$ and $x_{-\alpha}:=e_{j,i}$. Then $\{h_i,\ x_{\pm\alpha}\ |\ 1\leq i\leq n-1,\ \alpha\in R\}$ is a Chevalley basis for $\sl_n$. Given $i\in\{1,\ldots,n-1\}$, define $x_i:=x_{\alpha_i}=e_{i,i+1}$, $x_{-i}:=x_{-\alpha_i}=e_{i+1,i}$.  Note that, for all $1\leq i\leq n-1$, $\span\{x_{-i},h_i,x_i\}\cong\sl_2$.

Define nilpotent sub-superalgebras $\n^{\pm}:=\span\{x_{\alpha}|\alpha\in R^{\pm}\}$ and note that $\sl_n=\n^{-}\oplus\h\oplus\n^{+}.$

Define the set of fundamental weights $\{\omega_1,\ldots,\omega_{n-1}\}\subset\h^\ast$ by $\omega_i(h_j)=\delta_{i,j}$ for all $i,j\in\{1,\ldots,n-1\}$. Define $P^+:=\span_{\Z_{\geq0}}\{\omega_1,\ldots,\omega_{n-1}\}$ to be the set of dominant integral weights.

\subsection{Map Algebras and Weyl Modules}
For the remainder of this work fix a commutative, associative complex unital algebra $A$. Define the map algebra of $\sl_n$ to be $\sl_n\otimes A$ with Lie bracket given by linearly extending the bracket
$$[z\otimes a,w\otimes b]=[z,w]\otimes ab$$
for all $z,w\in\sl_n$ and $a,b\in A$.

Define $\bu(\sl_n\otimes A)$ to be the universal enveloping algebra of $\sl_n\otimes A$.

As in \cite{CFK} we define the global Weyl model for $\sl_n\otimes A$ of highest weight $\lambda\in P^+$ to be the module generated by a vector $w_\lambda$, called the highest weight vector, with relations:
$$(x\otimes a)w_\lambda=0,\hskip.5in (h\otimes 1)w_\lambda=\lambda(h)w_\lambda,\hskip.5in (x_{-i}\otimes 1)^{\lambda(h_i)+1}.w_\lambda=0$$
for all $a\in A$, $x\in\n^+$, $h\in\h$, and $1\leq i\leq n-1$.

\subsection{Multisets}

Given any set $S$ define a multiset of elements of $S$ to be a multiplicity function $\chi:S\to\Z_{\geq0}$. Define $\f(S):=\{\chi:S\to\Z_{\geq0}:|\supp\chi|<\infty\}$. For $\chi\in\f(S)$ define $|\chi|:=\sum_{s\in S}\chi(s)$. Notice that $\f(S)$ is an abelian monoid under function addition. For $\psi,\chi\in\f(S)$, $\psi\subseteq\chi$ if $\psi(s)\leq\chi(s)$ for all $s\in S$. Define $\f(\chi)(S):=\{\psi\in\f(S)\ |\ \psi\subseteq\chi\}$. In the case $S=A$ the $S$ will be omitted from the notation. So that $\f:=\f(A)$ and $\f(\chi)=\f(\chi)(A)$.

If $\psi,\chi\in \f$ with $\psi\in\f(\chi)$ we define $\chi-\psi$ by standard function subtraction. Also define $\pi:\f-\{0\}\to A$ by
$$\pi(\psi):=\prod_{a\in A}a^{\psi(a)}$$
and extend $\pi$ to $\f$ be setting $\pi(0)=1$. Define $\m:\f\to\Z$ by
$$\m(\psi):=\frac{|\psi|!}{\prod_{a\in A}\psi(a)!}$$
For all $\psi\in\f$, $\m(\psi)\in\Z$ because if $\supp\psi=\{a_1,\ldots,a_k\}$ then $\m(\psi)$ is the multinomial coefficient
$$\binom{|\psi|}{\psi(a_1),\ldots,\psi(a_k)}$$

For $s\in S$ define $\chi_s$ to be the characteristic function of the set $\{s\}$. Then for all $\chi\in\f(S)$
$$\chi=\sum_{s\in S}\chi(s)\chi_s$$

\subsection{The Symmetric Tensor Space}

Given any vector space $W$, there is an action of the symmetric group $S_k$ on $\underbrace{W \otimes W \otimes \cdots \otimes W}_\text{$k$-times}$ given by
$$\sigma(w_1\otimes w_2\otimes\cdots\otimes w_k)=v_{\sigma^{-1}(1)}\otimes v_{\sigma^{-1}(2)}\otimes\cdots\otimes v_{\sigma^{-1}(k)}\text{ where }v_1,\dots,v_k\in W.$$
For any vector space $W$, define its $k$th symmetric tensor space
$$S^k(W)=\span\left\{\sum_{\sigma \in S_k}\sigma(w_1\otimes\cdots\otimes w_k)\bigg|w_1,\dots,w_k\in W\right\}$$

Define $V\cong\C^n$ to be an $\sl_n$-module via left matrix multiplication, and write the basis as $v_1:=(1,0,\ldots,0)$, and for $i\in\{1,\ldots,n+m-1\}$, $v_{i+1}:=x_{-i}v_i$. Then $V\otimes A$ is an $\sl_n\otimes A$ module under the action $(z\otimes a)(w\otimes b)=zw\otimes ab$.

Given $\varphi_1,\dots,\varphi_n\in\f$ with $k:=\sum_{i=1}^n|\varphi_i|$ define
 $$w(\varphi_1, \dots, \varphi_n):=\bigotimes_{a_1\in\supp\varphi_1}(v_1\otimes a_1)^{\otimes\varphi_1(a_1)}\otimes\cdots\otimes\bigotimes_{a_n\in\supp\varphi_n}(v_n \otimes a_n)^{\otimes\varphi_n(a_n)}\in (V\otimes A)^{\otimes k}$$ and
 $$v(\varphi_1, \dots, \varphi_n) := \sum_{\sigma\in S_k}\sigma(w(\varphi_1, \dots, \varphi_n))\in S^k(V\otimes A).$$

We will need the following theorem of Feigin and Loktev, which is Theorem 6 in \cite{FL}.
\begin{thm}[Feigin--Loktev, 2004]\label{FL}
For all $m\in\N$
$W_A(m\omega_1) \cong S^m(V \otimes A)$ via the map given by
$$w_{m\omega_1} \mapsto (v_1 \otimes 1)^{ \otimes m}.$$
\end{thm}

We will also need the following lemma.

\begin{lem}\label{smbasis}
Let $\B$ be a basis for $A$. Then the set
$$\fB:=\left\{v(\varphi_1,\dots,\varphi_n)\ \bigg|\  \varphi_1,\ldots,\varphi_n\in\f(\B),\ \sum_{i=1}^n|\varphi_i|=m\right\}$$
is a basis for $S^m(V\otimes A)$.
\end{lem}
\begin{proof}
$\fB$ spans $S^m(V\otimes A)$ because $\B$ spans $A$ and $v_1,\ldots,v_n$ spans $V$. $\fB$ is linearly independent because the set
$$\{(v_{j_1}\otimes b_1)\otimes\dots\otimes(v_{j_m}\otimes b_m)\ |\ j_1,\ldots,j_m\in\{1,\ldots,n\},\ b_1,\ldots,b_m\in\B\}$$
is a basis for $(V\otimes A)^{\otimes m}$ and hence is linearly independent.
\end{proof}

Given $k\in\N$ define $\Delta^{k-1}: \bu(\sl_n \otimes A) \to \bu(\sl_n \otimes A)^{\otimes k}$ by extending the map $\sl_n\otimes A\to\bu(\sl_n\otimes A)^{\otimes k}$ given by

\begin{eqnarray*}
\Delta^{k-1}(z \otimes a)&=&\sum_{j = 0}^{k-1} 1^{\otimes j} \otimes (z \otimes a) \otimes 1^{\otimes k-1-j}
\end{eqnarray*}

Note that $\Delta^{k-1}(1)=1^{\otimes k}$ not $k1^{\otimes k}$.

Since $V\otimes A$ is a $\bu(\sl_n\otimes A)$ module, $(V\otimes A)^{\otimes m}$ is a left $\bu(\sl_n\otimes A)$-module with $u$ acting as $\Delta^{m-1}(u)$ followed by coordinatewise module actions. Moreover $S^m(V\otimes A)$ is a submodule under this action. Thus $S^m(V\otimes A)$ is a left $\bu(\sl_n\otimes A)$-module under this $\Delta^{m-1}$ action.

\subsection{}

 For all $i=1,\dots,n-1$ and $\chi,\varphi\in\f$ recursively define $q_i(\varphi,\chi)\in\bu(\sl_n\otimes A)$ as follows
\begin{eqnarray*}
 q_i(0, 0) &:=& 1\\
 q_i(0, \chi) &:=& -\frac{1}{|\chi |}\sum_{0\neq\psi\in\f(\chi)} \m(\psi)(h_i\otimes\pi(\psi))q_i(0,\chi-\psi)\\
q_i(\varphi,\chi)&:=&-\frac{1}{|\varphi |}\sum_{\psi\in\f(\chi)}\sum_{d\in\supp\varphi} \m(\psi)(x_{-i}\otimes d\pi(\psi))q_i(\varphi-\chi_d,\chi-\psi)
 \end{eqnarray*}

Given $\varphi_n, \dots, \varphi_n\in\f$, define
$$\hskip-.5in q(\varphi_1, \dots, \varphi_n)
:=q_{n-1}(\varphi_n,\varphi_{n-1})q_{n-2}((|\varphi_n|+|\varphi_{n-1}|)\chi_1, \varphi_{n-2})\dots q_{2}\left(\left(\sum_{j=3}^n|\varphi_j|\right)\chi_1, \varphi_{2}\right)q_{1}\left(\left(\sum_{k=2}^n|\varphi_j|\right)\chi_1, \varphi_{1}\right)$$
\begin{rem}
Note that the $q_i(0,\chi)$ coincide with the $p_i(\chi)$ defined in \cite{BC}.
\end{rem}

\section{Main Theorem}
The main result of this work is the theorem stated below.

\begin{thm}\label{thm}
Given a basis $\B$ for $A$ and $m\in\Z_{>0}$, the set
$$\left\{q(\varphi_1,\ldots,\varphi_n)w_{m\omega_1}\ \bigg|\ \varphi_1,\ldots,\varphi_n\in\f(\B),\ \sum_{i=1}^n|\varphi_i|=m\right\}$$
is a basis for $W_A(m\omega_1).$
\end{thm}

The proof of this theorem will be given after several necessary lemmas and propositions.

\subsection{Necessary Lemmas and Propositions}

\begin{prop}\label{deltak}
For all $k\in\N$ $\Delta^k=(1^{\otimes k-1}\otimes\Delta^1)\circ\Delta^{k-1}$.
\end{prop}
\begin{proof}
The case $k=1$ is trivial. For $k\geq2$ and $u\in\bu(\sl_n\otimes A)$ we have
\begin{eqnarray*}
\left(1^{\otimes k-1}\otimes\Delta^1\right)\left(\Delta^{k-1}(u)\right)&=&\left(1^{\otimes k-1}\otimes\Delta^1\right)\left(\sum_{j=0}^{k-1}1^{\otimes j}\otimes u\otimes1^{\otimes k-1-j}\right)\\
&=&\left(1^{\otimes k-1}\otimes\Delta^1\right)\left(\sum_{j=0}^{k-2}1^{\otimes j}\otimes u\otimes1^{\otimes k-1-j}+1^{\otimes k-1}\otimes u\right)\\
&=&\sum_{j=0}^{k-2}1^{\otimes j}\otimes u\otimes1^{\otimes k-2-j}\otimes\Delta^1(1)+1^{\otimes k-1}\otimes\Delta^1(u)\\
&=&\sum_{j=0}^{k-2}1^{\otimes j}\otimes u\otimes1^{\otimes k-2-j}\otimes1\otimes1+1^{\otimes k-1}\otimes(u\otimes1+1\otimes u)\\
&=&\sum_{j=0}^{k-2}1^{\otimes j}\otimes u\otimes1^{\otimes k-j}+1^{\otimes k-1}\otimes u\otimes1+1^{\otimes k-1}\otimes1\otimes u\\
&=&\sum_{j=0}^k1^{\otimes j}\otimes u\otimes1^{\otimes k-j}\\
&=&\Delta^k(u)
\end{eqnarray*}
\end{proof}

Given $\chi\in\f$ and $k\in\N$ define
$$\comp_k(\chi)=\left\{\psi:\{1,\ldots,k\}\to\f(\chi)\ \Bigg|\ \sum_{j=1}^k\psi(j)=\chi\right\}$$

\begin{lem}\label{deltaqi}
For all $i\in\{1,\ldots,n-1\}$
$$\Delta^{k-1}\left(q_i(\varphi, \chi)\right)=\sum_{\substack{\psi\in\comp_k(\chi)\\\phi\in\comp_k(\varphi)}}
q_i(\phi(1),\psi(1))\otimes\cdots\otimes q_i(\phi(k),\psi(k))$$
\end{lem}
\begin{proof}
This can be proven by induction on $k$. The case $k=1$ is trivial. In the case $k=2$ the lemma becomes
$$\Delta^1(q_i(\varphi, \chi)) = \sum_{\substack{\psi \in \f(\chi) \\ \phi \in \f(\varphi)}} q_i(\phi, \psi) \otimes q_i(\varphi - \phi, \chi - \psi)$$
This can be proven by induction on $|\varphi|$. For $k>2$ use Proposition \ref{deltak}. The details in the $\sl_2$ case can be found in \cite{Cham}. This can be extended to the $\sl_n$ case via the injection $\Omega_i:\sl_2\otimes A\to\sl_n\otimes A$ given by
$$\Omega_i(x ^-\otimes a)=x_{-i}\otimes a,\hskip.5in\Omega_i(h\otimes a)=h_i\otimes a,\hskip.5in\Omega_i(x^+\otimes a)=x_i\otimes a$$
For all $i\in\{1,\ldots,n-1\}$ and $a\in A$.
\end{proof}

\begin{lem}\label{action produces 0}
For all $\varphi,\chi \in \f$ with $|\varphi | + |\chi| > 1$ and all $i\in\{1,\ldots,n-1\}$ $ \ q_i(\varphi, \chi)(v_i \otimes 1) = 0$.
\end{lem}
\begin{proof}
Assume that $\varphi = 0$. This case will proceed by induction on $|\chi |>1$. If $ |\chi | = 2$ (so that $\chi=\{a,b\}$ for some $a,b\in A$) we have
 \begin{eqnarray*}
q_i(0,\{a, b\})(v_i \otimes 1) &=&\left[ (h_i \otimes a)\otimes(h_i \otimes b) - (h_i \otimes ab)\right](v_i \otimes 1) \\
 &=& (h_i \otimes a)\otimes (v_i \otimes b) - (v_i \otimes ab)\\
 &=& (v_i \otimes ab) - (v_i \otimes ab) \\
 &=& 0
\end{eqnarray*}
For the next case assume that $|\chi|>2$ then
$$q_i(0, \chi)(v_i \otimes 1) = -\frac{1}{|\chi |} \sum_{\emptyset \neq \psi \in \mathcal{F}(\chi)} {\mathcal{M}(\psi)(h_i \otimes \pi(\psi)) q_i(\chi - \psi)(v_i \otimes 1}) = 0$$
by induction.

Now assume that $|\varphi | = 1$ (or $\varphi = \chi_b$ for some $b\in A$). Then
\begin{eqnarray*}
q_i(\chi_b, \chi)(v_i \otimes 1) &=& -\sum_{\psi \in \f(\chi)} \m(\psi)(x_{-i} \otimes b \pi(\psi))  q_i(0, \chi-\psi)(v_i \otimes 1)\\
&=& - \m(\chi)(x_{-i} \otimes b\pi(\chi))(v_i \otimes 1)- \sum_{a\in \supp \chi} \m(\chi-\chi_a)(x_{-i} \otimes b \pi(\chi-\chi_a))  q_i(0, \chi_a)(v_i \otimes 1)\\
&=& - \m(\chi)(v_{i+1} \otimes b\pi(\chi)) - \sum_{a \in \supp \chi} \m(\chi - \chi_a)(x_{-i} \otimes b \pi(\chi - \chi_a))(-h_i \otimes a)(v_i \otimes 1)\\
&=& - \m(\chi)(v_{i+1} \otimes b\pi(\chi)) + \sum_{a \in \supp \chi} \m(\chi - \chi_a)(x_{-i} \otimes b \pi(\chi - \chi_a))(v_i \otimes a)\\
&=& - \m(\chi)(v_{i+1} \otimes b\pi(\chi)) + \sum_{a \in \supp \chi} \m(\chi - \chi_a)(v_{i+1} \otimes b\pi(\chi))\\
&=& - \m(\chi)(v_{i+1} \otimes b\pi(\chi)) + \sum_{a \in \supp \chi} \frac{(|\chi|-1)!}{\prod_{c \in \supp (\chi - \chi_a)} (\chi-\chi_a)(c)!} (v_{i+1} \otimes b\pi(\chi))\\
&=& - \m(\chi)(v_{i+1} \otimes b\pi(\chi)) + \sum_{a \in \supp \chi} \frac{(|\chi|-1)!}{\prod_{\substack{c \in \supp \chi \\ c\neq a}} \chi(c)!(\chi(a)-1)!} (v_{i+1} \otimes b\pi(\chi))\\
&=& - \m(\chi)(v_{i+1} \otimes b\pi(\chi)) + \sum_{a \in \supp \chi}\frac{\chi(a)}{|\chi|} \m(\chi)(v_{i+1} \otimes b\pi(\chi))\\
&=& - \m(\chi)(v_{i+1} \otimes b\pi(\chi)) + \m(\chi)(v_{i+1} \otimes b\pi(\chi))\\
&=& 0
\end{eqnarray*}

Finally assume that $|\varphi | > 1$. Then
\begin{eqnarray*}
q_i(\varphi, \chi)(v_i \otimes 1) &=& - \frac{1}{|\varphi|} \sum_{\psi \in \f(\chi)} \sum_{d\in \supp \varphi} \m(\psi)(x_{-i} \otimes d\pi(\psi))q_i(\varphi - \chi_d, \chi-\psi)(v_i \otimes 1)\\
&=& - \frac{1}{|\varphi |} \sum_{\psi \in \f(\chi)} \sum_{ d \in \supp \varphi} \m(\psi)\\
&& \Bigg( -\frac{1}{|\varphi | - 1} \sum_{\psi_1 \in \f(\chi-\psi)} \sum_{d_1 \in \supp(\varphi-\chi_d)} \m(\psi_1)  \\
&& (x_{-i} \otimes d \pi(\psi))(x_{-i} \otimes d_1 \pi(\psi_1)) q_i(\varphi-\chi_d-\chi_{d_1}, \chi - \psi - \psi_1)\Bigg)(v_i \otimes 1)\\
&=& 0
\end{eqnarray*}
because at least two $x_{-i}$ terms act on a single $v_i$ as 0.
\end{proof}

\begin{lem}\label{qivi}
For all $i\in\{1,\ldots,n-1\}$ and $\varphi,\chi\in\f$ with $|\varphi|+|\chi|=k$
$$q_i(\varphi,\chi)\left(v_i\otimes 1\right)^{\otimes k}=(-1)^kv\left(0,\ldots,0,\underbrace{\chi}_i,\underbrace{\varphi}_{i+1},0,\ldots,0\right)$$
\end{lem}
\begin{proof}
\begin{eqnarray*}
q_i(\varphi,\chi)\left(v_i\otimes 1\right)^{\otimes k}&=&\Delta^{k-1}\left(q_i(\varphi,\chi)\right)\left(v_i\otimes 1\right)^{\otimes k}\\
&=&\left(\sum_{\substack{\psi\in\comp_k(\chi)\\\phi\in\comp_k(\varphi)}}
q_i(\phi(1),\psi(1))\otimes\cdots\otimes q_i(\phi(k),\psi(k))\right)\left(v_i\otimes 1\right)^{\otimes k}\\
&&\text{by Lemma \ref{deltaqi}}\\
&=&\sum_{\substack{\psi\in\comp_k(\chi)\\\phi\in\comp_k(\varphi)}}
\left(q_i(\phi(1),\psi(1))\left(v_i\otimes 1\right)\right)\otimes\cdots\otimes \left(q_i(\phi(k),\psi(k))\left(v_i\otimes 1\right)\right)
\end{eqnarray*}
By Lemma \ref{action produces 0} we see that the only potentially nonzero terms in the sum are those for which $|\phi(j)|+|\psi(j)|\leq1$ for all $j\in\{1,\ldots,k\}$. Since $|\varphi|+|\chi|=k$ if we have $|\psi(j)|+|\phi(j)|=0$ for some $j\in\{1,\ldots,n-1\}$ then there is a $r\in\{1,\ldots,n-1\}$ such that $|\psi(r)|+|\phi(r)|>1$. So the only potentially nonzero terms in the sum are those for which $|\phi(j)|+|\psi(j)|=1$ for all $j\in\{1,\ldots,k\}$.
Suppose that $\phi(j)=\chi_a$ and $\psi(j)=0$ for some $j\in\{1,\ldots k\}$ and some $a\in A$. Then
$$q_i\left(\chi_a,0\right)\left(v_i\otimes 1\right)=-\left(x_{-i}\otimes a\right)\left(v_i\otimes 1\right)=-\left(v_{i+1}\otimes a\right)$$
Suppose that $\phi(j)=0$ and $\psi(j)=\chi_a$ for some $j\in\{1,\ldots k\}$ and some $a\in A$. Then
$$q_i\left(0,\chi_a\right)\left(v_i\otimes 1\right)=-\left(h_i\otimes a\right)\left(v_i\otimes 1\right)=-\left(v_i\otimes a\right)$$
So $-\left(v_{i+1}\otimes a\right)$ and $-\left(v_i\otimes a\right)$ are the only possibilities for factors in the tensor product above. Since we are summing over all possible submultisets of $\varphi$ and $\chi$ we have the result.
\end{proof}

\begin{lem}\label{qonv}
For all $m\in\mathbb{N}$ and all $\varphi_1,\ldots,\varphi_n\in\f$ with $\sum_{i=1}^n|\varphi_i|=m$
$$q(\varphi_1,\ldots,\varphi_n)(v_1\otimes 1)^{\otimes m} = (-1)^{\sum_{j=1}^nj|\varphi_j|}v(\varphi_1,\ldots,\varphi_n)$$
\end{lem}
\begin{proof}
Since for all $j\in\{1,\ldots,n-1\}$ and $k\in\{1,\ldots,n\}$
$$x_{-j}v_k=\delta_{j,k}v_{j+1},\hskip.5in h_jv_k=\delta_{j,k}v_j-\delta_{j+1,k}v_{j+1}$$
by Lemma \ref{qivi} we have\\
$q(\varphi_1,\ldots,\varphi_n)(v_1\otimes 1)^{\otimes m}$
\begin{eqnarray*}
&=&q_{n-1}(\varphi_n,\varphi_{n-1})q_{n-2}((|\varphi_n|+|\varphi_{n-1}|)\chi_1, \varphi_{n-2})\dots q_1\left(\left(\sum_{j=2}^n|\varphi_j|\right)\chi_1, \varphi_1\right)(v_1\otimes 1)^{\otimes m}\\
&=&(-1)^m q_{n-1}(\varphi_n,\varphi_{n-1})\dots q_2\left(\left(\sum_{j=3}^n|\varphi_j|\right)\chi_1, \varphi_2\right)v\left(\varphi_1,\left(\sum_{j=2}^n|\varphi_j|\right)\chi_1,0,\ldots,0\right)\\
&=&(-1)^{|\varphi_1|+2\sum_{j=2}^n|\varphi_j|}q_{n-1}(\varphi_n,\varphi_{n-1})\dots q_3\left(\left(\sum_{j=4}^n|\varphi_j|\right)\chi_1, \varphi_3\right)v\left(\varphi_1,\varphi_2,\left(\sum_{j=3}^n|\varphi_j|\right)\chi_1,0,\ldots,0\right)\\
&=&(-1)^{\sum_{j=1}^{n-2}j|\varphi_j|}q_{n-1}(\varphi_n,\varphi_{n-1})v\left(\varphi_1,\dots,\varphi_{n-2},\left(|\varphi_{n-1}|+|\varphi_n|\right)\chi_1,0\right)\\
&=&(-1)^{\sum_{j=1}^nj|\varphi_j|}v\left(\varphi_1,\dots,\varphi_n\right)
\end{eqnarray*}
\end{proof}

\subsection{The Proof of Theorem \ref{thm}}

\begin{proof}
By Lemmas  \ref{qonv} and \ref{smbasis}
$$\left\{q(\varphi_1,\dots,\varphi_n)(v_1\otimes 1)^{\otimes m}\ \bigg|\  \varphi_1,\ldots,\varphi_n\in\f(\B),\ \sum_{i=1}^n|\varphi_i|=m\right\}$$
is a basis for $S^m(V\otimes A)$. Therefore by Theorem \ref{FL}
$$\left\{q(\varphi_1,\dots,\varphi_n)w_{m\omega_1}\ \bigg|\  \varphi_1,\ldots,\varphi_n\in\f(\B),\ \sum_{i=1}^n|\varphi_i|=m\right\}$$
is a basis for $W_A(m\omega_1)$.
\end{proof}

\end{document}